\documentclass[12pt,  leqno]{article}
\usepackage[english]{babel}
\usepackage[T1]{fontenc}
\usepackage{amsopn,  amsthm,  amssymb,  amsmath,  epsfig,  graphics,  enumerate,  array,  calc}
\usepackage{txfonts}

\theoremstyle{plain}
 \newtheorem{thm}{Theorem}[section]
 
 \newtheorem{lem}[thm]{Lemma}
 \newtheorem{prop}[thm]{Proposition}
 \newtheorem{eg}[thm]{Example}
 \theoremstyle{definition}
 \newtheorem{defn}[thm]{Definition}
 \theoremstyle{remark}
 \newtheorem{rem}[thm]{Remark}

\begin{document}

\title{On a Conjecture about the Number of Solutions to Linear Diophantine
Equations with a Positive Integer Parameter }

\author{\large   Sheng CHEN$^{*}$ and Nan LI}

\maketitle

\begin{center}
{Department of Mathematics,    Harbin Institute of Technology,    \\
Harbin 150001,    P.R.China}\\

\end{center}

\footnotetext{\hspace{-0.5cm} $^*$\hspace*{1mm} Corresponding author.    
\hspace*{3mm}
 {\it E-mail address:} {schen@hit.edu.cn\\ Project 10526016  Supported by National Natural Science Foundation
of China and Project HITC200701 Supported by Science Research
Foundation in Harbin Institute of Technology.}}

\begin{center}
{\bf Abstract\\[1ex]}
\end{center}
Let $A(n)$ be a $k\times s$ matrix and $m(n)$ be a $k$ dimensional
vector,   where all entries of $A(n)$ and $m(n)$ are integer-valued
polynomials in $n$. Suppose that $$t(m(n)|A(n)
)=\#\{x\in\mathbb{Z}_{+}^{s}\mid A(n)x=m(n)\}$$ is finite for each
$n\in \mathbb{N}$,  where $Z_+$ is the set of nonnegative integers.
This paper conjectures that $t(m(n)|A(n))$ is an integer-valued
quasi-polynomial in $n$ for $n$ sufficiently large and verifies the
conjecture
in several cases.\\

\noindent 2000 AMS Classification:  Primary 05A15,   Secondary 11D45,  11P99\\
 Key words: \hspace{1mm}
integer-valued quasi-polynomial; generalized Euclidean division;

$\quad\quad\quad\, \, \, \, $ vector partition function

\section{Introduction}
Let $A(n)$ be a $k\times s$ matrix and $m(n)$ be a $k$ dimensional
vector,   where all entries of $A(n)$ and $m(n)$ are integer-valued
polynomials in $n$. Suppose that
$$t(m(n)|A(n))=\#\{x\in\mathbb{Z}_{+}^{s}\mid
A(n)x=m(n)\}$$  is   finite  for each $n\in \mathbb{N}$,  where
$Z_+$ denotes the set of nonnegative integers. In other words,
$t(m(n)|A(n))$ is the number of nonnegative integer solutions to
linear Diophantine equations with a positive integer parameter $n$.

We  conjecture that
\begin{center}
{\emph{ $t(m(n)|A(n))$ is an integer-valued quasi-polynomial in $n$
for $n$ sufficiently large. }}
\end{center}

 This conjecture is motivated by Ehrhart's conjecture,  for which the readers are referred to
\cite{E},  Exercise 12 in Chapter 4  of \cite{stanley} and its
errata and addenda.

 When $k=1$ and $s=3$,  we proved that the conjecture is true in
 $\cite{ours2}$. In this paper,  we prove the following theorems,
 which verify the conjecture in some cases.

\begin{thm} [$k=1$]\label{1.1} Let $m(n)$,  $a_i(n)$ ($1\leqslant i\leqslant s$) be integer-valued
polynomials in $n$ with positive leading coefficients. Suppose that
$A(n)=(a_1(n), \cdots, a_s(n))$  is strongly coprime (see Definition
$\ref{strong}$). Then $t(m(n)|A(n))$ (denoted by $p_{A(n)}(m(n))$ in
this case) is an integer-valued quasi-polynomial in $n$ for $n$
sufficiently large.
\end{thm}

\begin{thm}[$k>1$, Unimodular]\label{1.2}Suppose that $A(n)=A\in M_{k\times s}(\mathbb{Z})$  satisfies
the following  conditions: (1) $A$ is unimodular,  i.e.,  the
polyhedron $\{x:Ax=b, x\geqslant 0\}$ has only integral vertices
whenever $b$ is in the lattice spanned by the columns of $A$;
 (2) $Ker(A)\cap \mathbb{R}_{\geqslant 0}^{s}=0$. Then $t(m(n)|A(n))$ is a polynomial in $n$ for $n$ sufficiently
large.
\end{thm}

\begin{thm}[1-prime]\label{1.3} Suppose that $A(n)$ is a $2\times 3$
matrix and $A(n)$ is 1-prime (see Definition $\ref{1prime}$). Then
$t(m(n)|A(n))$ is a quasi-polynomial in $n$ for $n$ sufficiently
large.
\end{thm}

\section{The Theory of Generalized Euclidean Division and GCD}
As a preparation,  in this section,  we introduce the theory of
generalized Euclidean division and GCD for the ring $R$ of
integer-valued quasi-polynomials (see Definition $\ref{2.0.1}$ and
Proposition $\ref{R}$),  which was discussed in $\cite{mine}$. Here
we list related definitions and results without proofs.

\begin{defn}[Integer-valued quasi-polynomial,  see Definition 3 in $\cite{mine}$]\label{2.0.1}
We call a function $f:\mathbb{N}\to\mathbb{Z}$ an integer-valued
quasi-polynomial,  if there exists a positive integer $T$ and
polynomials $f_i(x)\in \mathbb{Z}[x]$ ($i=0, 1, \cdots, T-1$),  such
that when $n=Tm+i$ ($m\in \mathbb{N}$),  we have $f(n)=f_i(m)$. We
call $(T,  \{f_i(x)\}_{i=0}^{T-1})$ a \emph{representation} of
$f(x)$ and write
$$f(x)=(T,  \{f_i(x)\}_{i=0}^{T-1})$$ Then $\max\{degree(f_i(x))|
i=0,   1,   \cdots,   T-1\}$ and $T$  are called the  degree and
period
 of this representation respectively.
\end{defn}

\begin{prop}[Definition of the ring $R$,  see Proposition 6 in $\cite{mine}$]\label{R}The set of all integer-valued quasi-polynomials,    denoted
by $R$,    with pointwisely defined addition and multiplication,  is
a commutative ring with identity.
\end{prop}

\begin{defn}[See Definition 9 in $\cite{mine}$]\label{2. 0. 2}
               Let $r(x)\in R$ and $r(x)=(T,
 \{r_i(x)\}_{i=0}^{T-1})$. We shall say $r(x)$ is nonnegative and write $r(x)\succcurlyeq  0$,    if it
satisfies the following equivalent conditions:
\begin{enumerate}
\item[(a)]
for every $i=0,   1,   \cdots,   T-1$,  $r_i(x)=0$ or its leading
coefficient is positive;
          \item[(b)] there exists $C\in \mathbb{Z}$,    such that for every integer $n>C$,   we have $r(n) \geqslant 0$.
          \end{enumerate}
We shall say $r(x)$ is strictly positive and  write $r(x)\succ 0$,
if $r(x)=(T,
 \{r_i(x)\}_{i=0}^{T-1})$ satisfies the following  condition:

 $(a^{'})$ for every $i=0,   1,   \cdots,  T-1$,
the leading coefficient of $r_i(x)$ is positive.

 We write $
f(x)\preccurlyeq g(x)$  if $ g(x)-  f(x)\succcurlyeq 0$.
\end{defn}
\begin{defn}[See Definition 11 in $\cite{mine}$]\label{2. 0. 20}
Let $r(x) \in \mathbb{Z}[x] $,    define a function $| \cdot  |:$
$\mathbb{Z}[x]\rightarrow \mathbb{Z}[x]$ as follows:
$$|r(x)|=\left\{
          \begin{array}{ll}
            r(x),    & if\,   \,   \hbox{$r(x)\succ 0$} \\
            -r(x),    & if \,   \,    \hbox{$r(x)\prec 0$} \\
            0,    & if \,   \,   \hbox{$r(x)=0$}
          \end{array}
        \right. $$
\end{defn}

\begin{thm}[Generalized Euclidean division,  see Theorem 12 in $\cite{mine}$]\label{2. 0. 3}
Let $f(x),   g(x)\in \mathbb{Z}[x]$ and $g(x)\neq 0$.  Then there
exist unique $P(x),   r(x)\in R$ such that
$$f(x)=P(x)g(x)+r(x)\quad\,   where\quad
0\preccurlyeq r(x)\prec| g(x)|$$ In this situation,   we write
$P(x)=quo(f(x),  g(x))$ and $r(x)=rem(f(x), g(x))$. We call  the
inequality $0\leqslant r_1(x)<|g(x)|$ the remainder condition of
this division.\end{thm}
\begin{rem}[Similar to Remark 13 in $\cite{mine}$]\label{2. 0. 01}This division almost coincides with
the division in $\mathbb{Z}$ pointwisely in the following sense. Let
$f(x),   g(x)\in \mathbb{Z}[x]$ and
$$f(x)=P(x)g(x)+r(x), \quad\,   where\quad
0\preccurlyeq r(x)\prec| g(x)|$$By Definition $\ref{2. 0. 2}$,
the inequality $0\leqslant r(x)<|g(x)|$ will give an integer $C$
such that for all $n
>C$,
$$\big[\frac{f(n)}{g(n)}\big]=P(n)\,   ,   \,   \big\{\frac{f(n)}{g(n)}\big\}g(n)=r(n)$$
If $r(x)=0$,  we have
$$r(n)=\big\{\frac{f(n)}{g(n)}\big\}g(n)=0$$ for every $n\in \mathbb{N}$. The opposite also holds,
i.e.,  if  $r(n)=0$ for every $n\in \mathbb{N}$,  then $r(x)=0$ as
an integer-valued quasi-polynomial. So $ rem(f(x),   g(x))=0$ if and
only if for every $n\in \mathbb{N}$,   $ rem(f(n),   g(n))=0$.
\end{rem}

\begin{eg}[Similar to Example 14 in $\cite{mine}$]\label{eg1}The following is  an example to
illustrate the relation between Euclidean division in
$\mathbb{Z}[x]$ and in $\mathbb{Z}$.

 When $n>1$,
$$\bigg[\frac{n^{2}}{2n+1}\bigg]=\left\{
                             \begin{array}{ll}
                               m-1,    & \hbox{$n=2m$;} \\
                               m-1,    & \hbox{$n=2m-1$. }
                             \end{array}
                           \right. $$
$$\bigg\{\frac{n^{2}}{2n+1}\bigg\}(2n+1)=\left\{
                             \begin{array}{ll}
                               3m+1,    & \hbox{$n=2m$;} \\
                               m,    & \hbox{$n=2m-1$. }
                             \end{array}
                           \right.
$$
\end{eg}
 In order to use successive division in $R$,   we define generalized Euclidean
division for $f(x), g(x)\in R$ as in $\cite{mine}$. Suppose that
$T_0$ is the least common period of $f(x)$,  $g(x)$,  such that
$f(x)=(T_0,  \{f_i(x)\}_{i=0}^{T_0-1})$ and $g(x)=(T_0,
\{g_i(x)\}_{i=0}^{T_0-1})$.  Based on generalized division in
$\mathbb{Z}[x]$ (see Theorem $\ref{2. 0. 3}$),    we can define
$quo(f(x),   g(x))$ and $rem(f(x),  g(x))$ as follows (denoted by
$P(x)$ and $r(x)$ respectively): when $n=Tm+i$
$$P(n)=\left\{
         \begin{array}{ll}
           quo(f_i(m),   g_i(m)),    &  \hbox{ $g_i(m)\neq 0$} \\
           0,    &  \hbox{$g_i(m)=0$}
         \end{array}
       \right.,
r(n)=\left\{
         \begin{array}{ll}
           rem(f_i(m),   g_i(m)),    & \hbox{ $g_i(m)\neq 0$} \\
           f_i(m),    & \hbox{$g_i(m)=0$}
         \end{array}
       \right. $$
Then it is easy to check that $P(x),    r(x)\in R$ and
\begin{equation}\label{fg}
f(x)=R(x)g(x)+r(x)
\end{equation}
This will be called  the generalized Euclidean algorithm on the ring
of integer-valued quasi-polynomials.

By successive division in $R$,  we can develop generalized GCD
theory (see Definition $\ref{2. 0. 006}$),  similar to the case of
$\mathbb{Z}$ (see \cite{Bhu}).

\begin{defn}[Divisor,  similar to Definition 15 in $\cite{mine}$]\label{divisor}Suppose that $f(x),   g(x)\in R$ and for every $n\in \mathbb{N}$,
$g(n)\neq 0$.  Then by Remark $\ref{2. 0. 01}$,    the following two
conditions are equivalent:
\begin{enumerate}
  \item[(1)]rem($f(x),   g(x)$)=0;
  \item[(2)] for every $x\in \mathbb{N}$,    $g(x)$ is a divisor of $f(x)$.
\end{enumerate}
If  the two conditions are satisfied,    we shall call $g(x)$ a
divisor of $f(x)$ and write $g(x)\mid f(x)$.
\end{defn}
\begin{defn}[Quasi-rational function]\label{quasi-rational}
A  function $f:\mathbb{N}\to\mathbb{Q}$  is quasi-rational,  if
there exists a positive integer $T$ and rational functions
$\frac{f_i(x)}{g_i(x)}\in \mathbb{Q}(x)$,  where $f_i(x), g_i(x)\in
\mathbb{Z}[x]$ ($i=0, 1, \cdots, T-1$),  such that when $n=Tm+i$
($m\in \mathbb{N}$),  we have $f(n)=\frac{f_i(m)}{g_i(m)}$.
\end{defn}
By the equivalence of the  two conditions in Definition
$\ref{divisor}$,  we have the following property for quasi-rational
functions.

\begin{prop}\label{rational}
Let $f(x)$ be a  quasi-rational function. If for every $n\in
\mathbb{N}$,  $f(n)\in \mathbb{Z}$,  then $f(x)$ is an
integer-valued quasi-polynomial.
\end{prop}
Besides,  by the equivalence of the two conditions in Definition
$\ref{divisor}$,   we have the following proposition,   similar to
the situation in $\mathbb{Z}$.
 \begin{prop}[See Proposition 16 in $\cite{mine}$]\label{2. 0. 6}Let
$g(x),    f(x)\in $R.  If $f(x)\mid g(x)$ and $g(x)\mid f(x)$,    we
have $f(x)=\varepsilon g(x)$,    where $\varepsilon $ is an
invertible element in $R$ (see Proposition 7 in $\cite{mine}$ for
description of invertible elements in $R$).
\end{prop}

\begin{defn} [Generalized GCD,  see Definition 17 in $\cite{mine}$]\label{2. 0. 006}Let $f_1(x),   f_2(x),   \cdots,   f_s(x),   d(x)\\\in
R$.
\begin{enumerate}
  \item[(1)] We call $d(x)$  a
common divisor of $f_1(x),   f_2(x),   \cdots,   f_s(x)$,    if we
have $d(x)\mid f_k(x)$ for every $k=1,   2,   \cdots,   s$.
  \item[(2)]We call
$d(x)$  a greatest common divisor  of $f_1(x),   f_2(x),   \cdots,
f_s(x)$ if $d(x)$ is a common divisor  of $f_1(x),   f_2(x), \cdots,
f_s(x)$ and  for any common divisor $p(x)\in R$ of $f_1(x), f_2(x),
\cdots,  f_s(x)$,   we have $p(x)\mid d(x)$.
\end{enumerate}
\end{defn}
\begin{rem}[See Remark 18 in $\cite{mine}$]\label{2. 0. 014} Suppose that both $d_1(x)$ and   $ d_2(x)$  are
 greatest common divisors of $f_1(x),   f_2(x),  \cdots,   f_s(x)$.
  Then we have $d_1(x)\mid d_2(x)$ and $d_2(x)\mid d_1(x)$. Thus,  by
Proposition $\ref{2. 0. 6}$,    we have $d_1(x)=\varepsilon d_2(x)$,
where $\varepsilon $ is an invertible element in $R$.  So we have a
unique GCD $d(x)\in R$ for $f_1(x),  f_2(x),   \cdots,   f_s(x)$
such that $d(x)\succcurlyeq  0$ and write it as $ggcd(f_1(x),
f_2(x), \cdots,  f_s(x))$.
\end{rem}

\begin{lem}[Similar to  Lemma 20 in $\cite{mine}$]\label{2. 0. 008}Let $f_0(x),   g_0(x)\in R$.  By  generalized Euclidean division,    define
$f_k(x),   g_k(x)$  $(k\in \mathbb{N}-\{0\})$ recursively as
follows:
$$f_{k}(x)=g_{k-1}(x),   \,   \,   \,   \,   g_{k}(x)=rem(f_{k-1}(x),   g_{k-1}(x))$$
Then there exists $k_0\in \mathbb{N}-\{0\}$ such that
$rem(f_{k_0}(x),  g_{k_0}(x))$ is a constant.
\end{lem}
\begin{thm}[See Theorem 21 in $\cite{mine}$]\label{2. 0. 0099}Let $f_1(x),   f_2(x),   \cdots,   f_s(x)\in R$.
Then there exist $d(x),   u_i(x)\in R$ ($ i=1,   2,   \cdots,   s$),
such that $d(x)=ggcd(f_1(x),   f_2(x),   \cdots,   f_s(x))$ and
$$f_1(x)u_1(x)+f_2(x)u_2(x)+\cdots+f_s(x)u_s(x)=d(x)$$
\end{thm}
\begin{defn}\label{3.0.8}Let $a_1(x), a_2(x)$ be integer-valued quasi-polynomials in $x$.
Suppose that $ggcd(a(x), a_2(x))=1$. By Theorem $\ref{2. 0. 0099}$,
there exist two integer-valued quasi-polynomials $u_1(x), u_2(x)$,
such that $a_1(x)u_1(x)+a_2(x)u_2(x)=1$. Then we shall call $u_1(x)$
an \emph{inverse} of $a_1(x)$ mod $a_2(x)$ and $u_2(x)$ an
\emph{inverse} of $a_2(x)$ mod  $a_1(x)$,  denoted by $a_1(x)^{-1}$
and $a_2(x)^{-1}$ respectively.
\end{defn}

\section{Proof of Theorem $\ref{1.1}$}
\subsection{ A Lemma about Fourier-Dedekind Sum  } In this
subsection,  we suppose that $a_1, a_2, \cdots, a_s$ are pairwise
coprime positive integers. Let $A=(a_1, a_2, \cdots, a_s)$,  which
can be viewed as a $1\times s$ matrix. In this case,  we denote
$t(n|A)$ by $p_A(n)$,  i.e.,  $$p_A(n):=\#\{(x_1, x_2, \cdots,
x_s)\in \mathbb{Z}^{s}:\, \,  \, all\, \, x_j\geqslant 0, \, \,
x_1a_1+x_2a_2+\cdots+x_sa_s=n \}$$ In many papers,  it is also
written as $p_{\{a_1, a_2, \cdots, a_s\}}(n)$.
\begin{lem}\label{p_A}(see $\cite{Computing}$ and $\cite{poly}$ ). We have the following formula:
\begin{eqnarray*}
 p_A(n) &=& \bigg(\frac{B_1}{z-1}+\cdots+\frac{B_s}{(z-1)^s}+
\sum_{k=1}^{a_1-1}\frac{C_{1k}}{z-\xiup_{a_1}^{k}}+\cdots+ \sum_{k=1}^{a_s-1}\frac{C_{sk}}{z-\xiup_{a_s}^{k}}\bigg)\bigg|_{z=0}\\
   &=&
-B_1+B_2-\cdots+(-1)^sB_s+s_{-n}(a_2, a_3, \cdots, a_s;a_1)\\
&&+s_{-n}(a_1, a_3, \cdots, a_s;a_2)+\cdots+s_{-n}(a_1, a_2, \cdots,
a_{s-1};a_s)
\end{eqnarray*}
where
$$C_{ik}=\frac{1}{a_i}\sum_{k=1}^{a_i-1}\frac{1}{(1-\xiup_{a_i}^{ka_1})\cdots
(1-\xiup_{a_i}^{ka_{i-1}})(1-\xiup_{a_i}^{ka_{i+1}})\cdots(1-\xiup_{a_i}^{ka_s})\xiup_{a_i}^{kn}}$$
$$s_{-n}(a_1, \cdots,  a_{i-1}, a_{i+1}, \cdots,  a_s;a_i)=\frac{1}{a_i}\sum_{k=1}^{a_i-1}\frac{\xiup_{a_i}^{-kn}}{(1-\xiup_{a_i}^{ka_1})\cdots
(1-\xiup_{a_i}^{ka_{i-1}})(1-\xiup_{a_i}^{ka_{i+1}})\cdots(1-\xiup_{a_i}^{ka_s})}$$and
$B=\sum_{i=1}^{s}(-1)^{i}B_i$ is a polynomial in $n$ with degree
$s-1$.
\end{lem}
We call $s_{-n}(a_1, \cdots , a_{i-1}, a_{i+1}, \cdots,  a_s;a_i)$ a
Fourier-Dedekind sum and $B$ the polynomial part of $p_A(n)$.
\begin{lem}\label{4.1.0}Suppose that $a_1,  a_2,  \cdots,a_s$ are pairwise coprime
positive integers. We have the following identity for
Fourier-Dedekind sum.
\begin{eqnarray*}
   && s_{-n}(a_1,a_2,\cdots, a_{s-1};a_s) \\
  &=&
B^{'}-B-\sum_{i=1}^{s-1}s_{-\big(n-a_s\big(\big\{\frac{[\frac{n}{a_s}]}{a_i}\big\}a_i+1\big))}(a_1,\cdots,
a_{i-1},a_{i+1},\cdots ,a_s;a_i)
\end{eqnarray*} where
$B$ is the polynomial part of $p_{(a_1,\cdots,a_s)}(n)$,
$\sum_{i=1}^{s-1}(-1)^{i}B_i^{'}(t)$ is the polynomial part of
$p_{(a_1,\cdots,a_{s-1})}(n-a_st)$ and
$$B^{'}=\sum_{t=0}^{\big[\frac{n}{a_s}\big]}\sum_{i=1}^{s-1}(-1)^{i}B_i^{'}(t)$$Besides, $B^{'}-B$ is a quasi-polynomial in
$n$.
\end{lem}
\begin{proof}
By Lemma $\ref{p_A}$, for any $t\in \mathbb{Z}$ and $0\leqslant
t\leqslant \big[\frac{n}{a_s}\big]$,
$\sum_{i=1}^{s-1}(-1)^{i}B_i^{'}(t)$ is a polynomial in $n-a_st$
with degree $s-2$. Note that $\sum_{t=1}^{n}t^{s}$ is a polynomial
in $n$. So we have
$B^{'}=\sum_{t=0}^{\big[\frac{n}{a_s}\big]}\sum_{i=1}^{s-1}(-1)^{i}B_i^{'}(t)$
is a quasi-polynomial in $n$. Thus $B^{'}-B$ is a quasi-polynomial
in $n$. Note that
\begin{equation}\label{RL}
    p_{(a_1,a_2,\cdots,a_s)}(n)=\sum_{t=0}^{\big[\frac{n}{a_s}\big]}p_{(a_1,a_2,\cdots,a_{s-1})}(n-a_st)
\end{equation}

By the formula for $p_A(n)$ in Lemma $\ref{p_A}$, we have
$$p_{(a_1,a_2,\cdots,a_s)}(n)= B+\sum_{i=1}^{s}s_{-n}(a_1,\cdots
a_{i-1},a_{i+1},\cdots, a_{s-1},a_s;a_i)$$and
\begin{eqnarray*}
  \sum_{t=0}^{\big[\frac{n}{a_s}\big]}p_{(a_1,a_2,\cdots,a_{s-1})}(n-a_st)&=&
B^{'}+\sum_{t=0}^{\big[\frac{n}{a_s}\big]}\sum_{i=1}^{s-1}s_{-(n-a_st)}(a_1,\cdots,
a_{i-1},a_{i+1},\cdots, a_{s-1};a_i) \\
   &=&
B^{'}+\sum_{i=1}^{s-1}\sum_{t=0}^{\big[\frac{n}{a_s}\big]}s_{-(n-a_st)}(a_1,\cdots,
a_{i-1},a_{i+1},\cdots, a_{s-1};a_i) \\
   &=& B^{'}+\sum_{i=1}^{s-1}F_i
\end{eqnarray*}
where
$$
           F_i
=\sum_{t=0}^{\big[\frac{n}{a_s}\big]}s_{-(n-a_st)}(a_1,\cdots,
a_{i-1},a_{i+1},\cdots, a_{s-1};a_i)$$
 Note that for any $b\in \mathbb{Z}$ and
$1\leqslant k\leqslant a_i-1$,
$$\sum_{t=b}^{b+a_i-1}\xiup_{a_i}^{ka_st}=0$$  Besides, for $i=1,2,\cdots,s-1$,
we have
$$\big[\frac{n}{a_s}\big]-\bigg[\frac{\big[\frac{n}{a_s}\big]}{a_i}\bigg]a_i
=\bigg\{\frac{\big[\frac{n}{a_s}\big]}{a_i}\bigg\}a_i$$ Therefore
\begin{eqnarray*}
   \sum_{t=0}^{\big[\frac{n}{a_s}\big]}\xiup_{a_i}^{ka_st} &=& \sum_{t=0}^{\bigg\{\frac{\big[\frac{n}{a_s}\big]}{a_i}\bigg\}a_i}
\xiup_{a_i}^{ka_st} =
   \frac{1-\xiup_{a_i}^{ka_s\bigg(\bigg\{\frac{\big[\frac{n}{a_s}\big]}{a_i}\bigg\}a_i+1\bigg) }}{1-\xiup_{a_i}^{ka_{s}}}
\end{eqnarray*}
Thus we can see that
\begin{eqnarray*}
  F_i &=& \frac{1}{a_i}\sum_{t=0}^{\big[\frac{n}{a_s}\big]}\sum_{k=1}^{a_i-1}\frac{\xiup_{a_i}^{-k(n-a_st)}}{(1-\xiup_{a_i}^{ka_1})\cdots
(1-\xiup_{a_i}^{ka_{i-1}})(1-\xiup_{a_i}^{ka_{i+1}})\cdots(1-\xiup_{a_i}^{ka_{s-1}})}\\
&=&
 \frac{1}{a_i}\sum_{k=1}^{a_i-1}\frac{\xiup_{a_i}^{-kn}\sum_{t=0}^{\big[\frac{n}{a_s}\big]}\xiup_{a_i}^{ka_st}}
{(1-\xiup_{a_i}^{ka_1})\cdots
(1-\xiup_{a_i}^{ka_{i-1}})(1-\xiup_{a_i}^{ka_{i+1}})\cdots(1-\xiup_{a_i}^{ka_{s-1}})}\\
   &=&  \frac{1}{a_i}\sum_{k=1}^{a_i-1}\frac{\xiup_{a_i}^{-kn}-\xiup_{a_i}^{ka_s\bigg(\bigg\{\frac{\big[\frac{n}{a_s}\big]}{a_i}\bigg\}a_i+1\bigg)    -kn}}
{(1-\xiup_{a_i}^{ka_1})\cdots
(1-\xiup_{a_i}^{ka_{i-1}})(1-\xiup_{a_i}^{ka_{i+1}})\cdots(1-\xiup_{a_i}^{ka_{s-1}})(1-\xiup_{a_i}^{ka_s})}\\
   &=&s_{-n}(a_1,\cdots,
a_{i-1},a_{i+1},\cdots, a_{s-1},a_s;a_i)\\
&&-s_{-\big(n-a_s\big(\big\{\frac{[\frac{n}{a_s}]}{a_i}\big\}a_i+1\big))}(a_1,\cdots,
a_{i-1},a_{i+1},\cdots, a_s;a_i)
\end{eqnarray*}
Therefore, we have
\begin{eqnarray*}\label{sum}
   \sum_{t=0}^{\big[\frac{n}{a_s}\big]}p_{(a_1,a_2,\cdots,a_{s-1})}(n-a_st)
&=& B^{'}+\sum_{i=1}^{s-1}\bigg(s_{-n}(a_1,\cdots,
a_{i-1},a_{i+1},\cdots,
a_{s-1},a_s;a_i)\\&&-s_{-\big(n-a_s\big(\big\{\frac{[\frac{n}{a_s}]}{a_i}\big\}a_i+1\big))}(a_1,\cdots,
a_{i-1},a_{i+1},\cdots, a_s;a_i)\bigg)
\end{eqnarray*}

Then by ($\ref{RL}$), we have
\begin{eqnarray*}
  &&B+\sum_{i=1}^{s}s_{-n}(a_1,\cdots a_{i-1},a_{i+1},\cdots,
a_{s-1},a_s;a_i) \\&=&B^{'}+\sum_{i=1}^{s-1}\bigg(s_{-n}(a_1,\cdots,
a_{i-1},a_{i+1},\cdots,
a_{s-1},a_s;a_i)\\
&&-s_{-\big(n-a_s\big(\big\{\frac{[\frac{n}{a_s}]}{a_i}\big\}a_i+1\big))}(a_1,\cdots,
a_{i-1},a_{i+1},\cdots, a_s;a_i)\bigg)
\end{eqnarray*}
Now the result follows easily.
\end{proof}

\subsection{An Application of the Lemma }
In this subsection, we apply Lemma $\ref{4.1.0}$ to prove Theorem
$\ref{1.1}$.

Let $m(n)$, $a_i(n)$ ($1\leqslant i\leqslant s$) be integer-valued
polynomials in $n$ with positive leading coefficients. Suppose that
$A(n)=(a_1(n),a_2(n),\cdots,a_s(n))$, where
$a_1(n),a_2(n),\cdots,a_s(n)$ are pairwise coprime as integer-valued
quasi-polynomials, i.e., $ggcd(a_1(n),a_2(n),\cdots,a_s(n))=1$ (see
Remark $\ref{2. 0. 014}$ for the notation of $ggcd$). By Definition
$\ref{divisor}$, this means that
$gcd(a_1(n),\\a_2(n),\cdots,a_s(n))=1$ for any $n\in \mathbb{N}$. In
this case we denote $t(m(n)|A(n))$ by $p_{A(n)}(m(n))$, i.e.,
$$p_{A(n)}(m(n)):=\#\{(x_1,x_2,\cdots,x_s)\in \mathbb{Z}^{s}:\,\, \,all\,\,x_j\geqslant 0,\,\,x_1a_1(n)+x_2a_2(n)+\cdots+x_sa_s(n)=m(n) \}$$
From Lemma $\ref{p_A}$ and
Lemma $\ref{4.1.0}$, we have the following lemmas.
\begin{lem}\label{p_A(n)1}
Suppose that $A(n)=(a_1(n),a_2(n),\cdots,a_s(n))$, where
$a_1(n),a_2(n),\cdots,a_s(n)$ are pairwise coprime integer-valued
polynomials with positive leading coefficients. Then
\begin{eqnarray*}
  p_{A(n)}(m(n)) &=&
B+s_{-m(n)}(a_2(n),a_3(n),\cdots,a_s(n);a_1(n))\\
&&+s_{-m(n)}(a_1(n),a_3(n),\cdots,a_s(n);a_2(n))+\cdots\\&&+s_{-m(n)}(a_1(n),a_2(n),\cdots,a_{s-1}(n);a_s(n))\label{p_A(n)}
\end{eqnarray*}

Where $B=-B_1+B_2-\cdots+(-1)^sB_s$ is a quasi-polynomial in $n$ and
for $1\leqslant i \leqslant s$
\begin{eqnarray*}
  && s_{-m(n)}(a_1(n),\cdots, a_{i-1}(n),a_{i+1}(n),\cdots, a_s(n);a_i(n)) \\
  &=& \frac{1}{a_i(n)}
\sum_{k=1}^{a_i(n)-1}\frac{\xiup_{a_i(n)}^{-km(n)}}{(1-\xiup_{a_i(n)}^{ka_1(n)})\cdots
(1-\xiup_{a_i(n)}^{ka_{i-1}(n)})(1-\xiup_{a_i(n)}^{ka_{i+1}(n)})\cdots(1-\xiup_{a_i(n)}^{ka_s(n)})}
\end{eqnarray*}
\end{lem}

\begin{lem}\label{4.1.01} Suppose that $a_1(n),a_2(n),\cdots,a_s(n)$ are pairwise coprime integer-valued
polynomials with positive leading coefficients. Then
\begin{eqnarray*}
   && s_{-m(n)}(a_1(n),a_2(n),\cdots, a_{s-1}(n);a_s(n)) \\
  &=&
B^{'}-B-\sum_{i=1}^{s-1}s_{-\bigg(m(n)-a_s(n)\big(\big\{\frac{[\frac{m(n)}{a_s(n)}]}{a_i(n)}\big\}a_i(n)+1\big)\bigg)}(a_1(n),\cdots,
a_{i-1}(n),a_{i+1}(n),\cdots ,a_s(n);a_i(n))
\end{eqnarray*}
where $B$ is the polynomial part of
$p_{(a_1(n),\cdots,a_s(n))}(m(n))$,
$\sum_{i=1}^{s-1}(-1)^{i}B_i^{'}(t)$ is the polynomial part of
$p_{(a_1(n),\cdots,a_{s-1}(n))}(m(n)-a_s(n)t)$ and
$$B^{'}=\sum_{t=0}^{\big[\frac{n}{a_s}\big]}\sum_{i=1}^{s-1}(-1)^{i}B_i^{'}(t)$$Besides $B^{'}-B$ is a quasi-polynomial in
$n$.
\end{lem}

Let all entries of vector $A(n)$ be integer-valued quasi-polynomials
in $n$. Then we say  $A(n)$ is constant if the degree (see
Definition $\ref{2.0.1}$) of each nonzero entrie in $A(n)$ is zero.
Besides, we say  a vector is pairwise coprime if  entries of that
 vector are pairwise coprime.

For $A(n)=(a_1(n),a_2(n),\cdots,a_s(n))$, define
$$A^{(0)}(n)=\{a_1(n),a_2(n),\cdots,a_s(n)\}$$ If $A^{(k)}(n)$ is
not empty, define
$$a_{i_1i_2\cdots i_ki_{k+1}i}(n)=\left\{
                                \begin{array}{ll}
                                  \bigg\{\frac{a_{i_1i_2\cdots i_ki}(n)}{a_{i_1i_2\cdots i_ki_{k+1}}(n)}\bigg\}a_{i_1i_2\cdots i_ki_{k+1}}(n),
& \hbox{$i\neq i_{k+1}$;} \\
                                  a_{i_1i_2\cdots i_ki_{k+1}}(n), & \hbox{$i=i_{k+1}$.}
                                \end{array}
                              \right.
$$
$$A_{i_1i_2\cdots i_ki_{k+1}}(n)=\big(a_{i_1i_2\cdots i_ki_{k+1}1}(n),\cdots,a_{i_1i_2\cdots i_ki_{k+1}s}(n)\big)$$
and $$A^{(k+1)}(n)=\bigg\{A_{i_1i_2\cdots i_ki_{k+1}}(n)\mid
i_1,i_2,\cdots,i_{k+1}\in \{1,2,\cdots,s\},A_{i_1i_2\cdots
i_ki_{k+1}}(n) \,\,is\,\, not\,\, constant\bigg\}$$ If $A^{(k)}(n)$
is empty, then let $A^{(l)}(n)$ be empty for any $l>k$. By Lemma
$\ref{2. 0. 008}$, we have
\begin{lem}\label{4.4}There exists a unique positive
integer $h$ such that  $A^{(l)}(n)$ is empty for each $l\geqslant h$
while $A^{(k)}(n)$ is not empty  for each $0\leqslant k<h$.
\end{lem}
\begin{defn}[Strongly coprime]\label{strong}Let $h$ be the integer
satisfying conditions in Lemma $\ref{4.4}$. We say
$A(n)=(a_1(n),a_2(n),\cdots,a_s(n))$ is strongly coprime if elements
of $A^{(0)}(n),A^{(1)}(n),\\ \cdots, A^{(h-1)}(n)$ are all pairwise
coprime.
\end{defn}

\begin{proof}[\textbf{Proof of Theorem $\ref{1.1}$}]First note that when $a_{i_0}$ is a constant,
we have
\begin{eqnarray*}
  &&  s_{-m(n)}(a_1(n),\cdots, a_{i_0-1}(n),a_{i_0+1}(n),\cdots, a_s(n);a_{i_0})\\
 &=& \frac{1}{a_{i_0}}
\sum_{k=1}^{a_{i_0}-1}\frac{\xiup_{a_{i_0}}^{-km(n)}}{(1-\xiup_{a_{i_0}}^{ka_1(n)})\cdots
(1-\xiup_{a_{i_0}}^{ka_{i_0-1}(n)})(1-\xiup_{a_{i_0}}^{ka_{i_0+1}(n)})\cdots(1-\xiup_{a_{i_0}}^{ka_s(n)})}\in
\mathbb{Q} \end{eqnarray*}  for any $n\in \mathbb{N}$ and
$s_{-m(n)}(a_1(n),\cdots, a_{i_0-1}(n),a_{i_0+1}(n),\cdots,
a_s(n);a_{i_0})$ is quasi-rational (see Definition
$\ref{quasi-rational}$) as a function in $n$.

Since $A(n)$ is  strongly coprime, by Lemma $\ref{4.4}$ and
Definition $\ref{strong}$, we can apply Lemma $\ref{4.1.01}$
successively to the formula for $p_{A(n)}(m(n))$ in Lemma
$\ref{p_A(n)1}$ until $p_{A(n)}(m(n))$ can be expressed as a sum of
finitely many quasi-rational functions. Then by Proposition
$\ref{rational}$, $p_{A(n)}(m(n))$ is an integer-valued
quasi-polynomial for $n$ sufficiently large.
\end{proof}
\begin{eg}Suppose that $a_1,a_2,\cdots, a_{s-1}$ are positive integers,
$a_s(n)$ is an integer-valued polynomial in $n$ with positive
leading coefficients and $a_1,\cdots, a_{s-1},a_s(n)$ are pairwise
coprime. Then by  Definition $\ref{strong}$, $A(n)=(a_1,\cdots,
a_{s-1},a_s(n))$ is strongly coprime. Thus by Theorem $\ref{1.1}$,
$p_{A(n)}(m(n))$ is an integer-valued quasi-polynomial in $n$ for
$n$ sufficiently large.
\end{eg}

\section{ Proof of Theorem $\ref{1.2}$ and Theorem $\ref{1.3}$ }
In the previous section we have studied the conjecture in case of
one equation. Now we turn to the case of equations and prove Theorem
$\ref{1.2}$ and Theorem $\ref{1.3}$.
\subsection{ Proof of Theorem $\ref{1.2}$ }
Suppose that $A(n)=A=(a_1, a_2,\cdots, a_s)\in M_{k\times
s}(\mathbb{Z})$  satisfies the following  conditions: (1) $A$ is
unimodular, i.e., the polyhedron $\{x:Ax=b,x\geqslant 0\}$ has only
integral vertices whenever $b$ is in the lattice spanned by the
columns of $A$; (2) $Ker(A)\cap \mathbb{R}_{\geqslant 0}^{s}=0$.

Define $pos(A)=\{\sum_{i=1}^{s}\lambdaup_ia_i\in
\mathbb{R}^{k}:\lambdaup_1,\cdots,\lambdaup_s\geqslant 0\}$.  For
$\sigma \subset [s]:=\{1,\cdots,s\}$, we consider the submatrix
$A_{\sigma}:=(a_i:i\in\sigma)$, the polyhedral cone
$pos(A_{\sigma})$, and the abelian group $\mathbb{Z}A_{\sigma}$
spanned by the columns of $A_{\sigma}$. Since $A$ is unimodular, $A$
is surjective, that is, $\mathbb{Z}A=\mathbb{Z}^{k}$. This implies
that the semigroup $\mathbb{N}A:=pos(A)\cap \mathbb{Z}A$ is
saturated. A subset $\sigma$ of $[s]$ is a basis if
$\#(\sigma)=rank(A_{\sigma})=k$. The chamber complex is the
polyhedral subdivision of the cone $pos(A)$ which is defined as the
common refinement of the simplicial cones $pos(A_{\sigma})$, where
$\sigma$ runs over all bases. Each chamber $C$, i.e.,  maximal cell
in the chamber complex, is indexed by the set $\bigtriangleup
(C)=\{\sigma\subset [s]:C\subseteq pos(A_{\sigma})\}$.

\begin{lem}[See \cite{stum},Theorem 1.1 or \cite{number}, Corollary 3.1]\label{uni}
Under the conditions of $A$ in Theorem \ref{1.2}, the vector
partition function $$\phi_{A}(u)=\#\{x:Ax=b,x\geqslant 0,x\,\,is
\,\,integral\}$$ is a polynomial function of degree $s-k$ in
$u=(u_1,\cdots,u_k)$ on each chamber.
\end{lem}

\begin{proof}[\textbf{Proof of Theorem $\ref{1.2}$}] Each  chamber for $A$,
as a convex polyhedral cone, can be described  by linear
inequalities in $u=(u_1,\cdots,u_k)$. Since these linear
inequalities are independent of $n$, there exists a chamber  $C$
such that
 $ m(n)$ lies in $C$ for  $n$ sufficiently large. Thus the result follows from
Lemma $\ref{uni}$.
\end{proof}
\begin{rem}
By Theorem 1 in \cite{vector}, using the idea in the proof of
Theorem $\ref{1.2}$, we can verify the conjecture in  case of
$A(n)=A\in M_{k\times s}(\mathbb{Z})$ without  unimodular condition.

\end{rem}
 Now we give an example to illustrate Theorem
$\ref{1.2}$.

 \begin{eg}
 Let
 $$A=\left(
       \begin{array}{ccccc}
         1 & 0 & 0 & 1 & 1 \\
         0 & 1 & 0 & 1 & 0 \\
         0 & 0 & 1 & 0 & 1 \\
       \end{array}
     \right)
 $$
 Then $A$ is a   unimodular matrix.
 By  $\cite{stum}$, we have
$$\phi_A(a,b,c)=\left\{
                                \begin{array}{ll}
                                  bc+b+c+1, & \hbox{if $\,\,(a,b,c)\in \Omega_1$} \\
                                  \frac{1}{2}a^{2}+\frac{3}{2}a+1, & \hbox{if\,\, $(a,b,c)\in \Omega_2$} \\
                                  ab-\frac{1}{2}b^{2}+\frac{1}{2}b+a+1, & \hbox{if\,\, $(a,b,c)\in \Omega_3$} \\
                                  ac-\frac{1}{2}c^{2}+\frac{1}{2}c+a+1, & \hbox{if \,\,$(a,b,c)\in \Omega_4$} \\
                                  ab+ac-\frac{1}{2}(a^{2}+b^{2}+c^{2})+\frac{1}{2}(a+b+c)+1, & \hbox{if \,\,$(a,b,c)\in \Omega_5$}
                                \end{array}
                              \right.
 $$
where
\begin{eqnarray*}
  \Omega_1 &=& \{(a,b,c)|a\geqslant b+c \,\, and \,\, b,c\geqslant0\} \\
  \Omega_2 &=& \{(a,b,c)|min\{b,c\}\geqslant a\geqslant 0\} \\
 \Omega_3 &=& \{(a,b,c)|c\geqslant a\geqslant b\geqslant 0\} \\
 \Omega_4 &=& \{(a,b,c)|b\geqslant a\geqslant c\} \\
\Omega_5 &=& \{(a,b,c)|b+c\geqslant a\geqslant max\{b,c\}\} \\
\end{eqnarray*}
If we take $A(n)=A$  and $$m(n)=\left(
                           \begin{array}{c}
                             2n^{2}-n \\
                             2n^{2}+5\\
                             n^{2}+10n \\
                           \end{array}
                         \right)$$ then $m(n)$ lies in $\Omega_4$ for $n$ sufficiently
                         large. By the previous formula for
                         $\phi_A(a,b,c)$, we have, for $n>11$
                         $$t(m(n)|A(n))=\frac{3}{2}n^{4}+9n^{3}-\frac{115}{2}n^{2}+4n+1$$
 \end{eg}
\subsection{Proof of Theorem $\ref{1.3}$}

\begin{defn}[Similar to \cite{Xu}, p.79]\label{1prime}We call  a polynomial matrix $A(n)_{k\times (k+1)}$ 1-prime  if
$ggcd\{|det(Y)|:Y\in \mathcal {B}\}=1$ (see Remark $\ref{2. 0. 014}$
for the notation of $ggcd$), where $\mathcal {B}$ denotes the set of
all the $k\times k$ submatrix of $A(n)$.

Let $$A= \left(
  \begin{array}{cccc}
    x_1 & x_2  & x_3 \\
    y_1 & y_2  & y_3 \\
  \end{array}
\right)\in M_{2\times 3}(\mathbb{Z})
$$For $i=1,2$, suppose that $\frac{y_i}{x_i}<\frac{y_{i+1}}{x_{i+1}}$, and define
$$\Omega_i=\bigg\{(x,y)^{T}|(x,y)^{T}\in pos(A),\frac{y_i}{x_i}<\frac{y}{x}<\frac{y_{i+1}}{x_{i+1}}\bigg\}$$
\end{defn}

\begin{lem}[See Theorem 7 in \cite{Xu2}]\label{5.4}Let
$$A=\left(
      \begin{array}{ccc}
        x_1 & x_2 & x_3 \\
        y_1 & y_2 & y_3 \\
      \end{array}
    \right)\in M_{2\times 3}(\mathbb{Z})
$$be a 1-prime matrix. Let $$M_{ij}=\left(
                   \begin{array}{cc}
                     x_i & x_j \\
                     y_i & y_j \\
                   \end{array}
                 \right)
$$ and let $Y_{ij}=det(M_{ij})$. Then we have the following formula.

When $m=(m_1,m_2)^{T}\in \overline{\Omega}_1\cap \mathbb{Z}^{2}$,
\begin{eqnarray*}
   t(m|A)&=& \frac{m_2x_1-m_1y_1}{Y_{12}Y_{13}}-
\{\frac{(f_{12}Y_{13}+g_{12}Y_{23})^{-1}(m_2(f_{12}x_1+g_{12}x_2)-m_1(f_{12}y_1+g_{12}y_2))}{Y_{12}}\}
 \\
  &&-\{\frac{(f_{13}Y_{12}+g_{13}Y_{23})^{-1}(m_2(f_{13}x_1+g_{13}x_3)-m_1(f_{13}y_1+g_{13}y_3))}{Y_{13}}\}+1
\end{eqnarray*}
When $m=(m_1,m_2)^{T}\in \overline{\Omega}_2\cap \mathbb{Z}^{2}$,
\begin{eqnarray*}
   t(m|A)&=& \frac{m_1y_3-m_2y_3}{Y_{23}Y_{13}}-
\{\frac{(f_{23}Y_{13}+g_{23}Y_{12})^{-1}(m_1(f_{23}x_3+g_{23}x_2)-m_2(f_{23}y_3+g_{23}y_2))}{Y_{23}}\}
 \\
  &&-\{\frac{(f_{13}Y_{12}+g_{13}Y_{23})^{-1}(m_1(f_{13}x_1+g_{13}x_3)-m_2(f_{13}y_1+g_{13}y_3))}{Y_{13}}\}+1
\end{eqnarray*}
where, $f_{12},g_{12},f_{13},g_{13},f_{23}$ and $g_{23}\in
\mathbb{Z}$ satisfy $gcd(f_{12}Y_{13}+g_{12}Y_{23},Y_{12})=1$,
$gcd(f_{13}Y_{12}+g_{13}Y_{23},Y_{13})=1$ and
$gcd(f_{23}Y_{13}+g_{23}Y_{12},Y_{23})=1$, moreover,
$(f_{12}Y_{13}+g_{12}Y_{23})^{-1}(f_{12}Y_{13}+g_{12}Y_{23})\equiv
1$ $mod$ $Y_{12}$,
$(f_{13}Y_{12}+g_{13}Y_{23})^{-1}(f_{13}Y_{12}+g_{13}Y_{23})\equiv
1$ $mod$ $Y_{13}$,
$(f_{23}Y_{13}+g_{23}Y_{12})^{-1}(f_{23}Y_{13}\\+g_{23}Y_{12})\equiv
1$ $mod$ $Y_{23}$.
\end{lem}

Now we can  prove Theorem $\ref{1.3}$.

\begin{proof}[\textbf{Proof of Theorem $\ref{1.3}$}]  Suppose that $A(n)$ satisfies conditions in
Theorem $\ref{1.3}$. By the theory of generalized Euclidean division
and GCD in Section 2, for every $n\in\mathbb{N}$, $A(n)$ is 1-prime.
Note that there exists $i=1$ or $2$ such that $m(n)$ lies in
$\overline{\Omega}_i$ for $n$ sufficiently large. Thus, by Theorem
$\ref{2. 0. 3}$, Proposition $\ref{rational}$ and Lemma $\ref{5.4}$,
$t(m(n)|A(n))$ is an integer-valued quasi-polynomial in $n$ for $n$
sufficiently large.
\end{proof}
\begin{eg}\label{5.2.1}Let  $A(n)=\left(
  \begin{array}{ccc}
    2n+1  & 3n+1& n^{2} \\
    2 & 3& n+1  \\
  \end{array}
\right)$ and $m(n)=\left(
          \begin{array}{c}
            3n^{3}+1 \\
            3n^{2}+n-1 \\
          \end{array}
        \right)$. Then we can check that $A(n)$  satisfies conditions in Theorem
$\ref{1.3}$ and $m(n)\in \Omega_2$  for $n$ sufficiently large.
\end{eg}

\begin{center}
{\bf Acknowledgements\\[1ex]}
\end{center}
The second author are grateful to Prof. R. Stanley, M. Beck and Z.
Xu for communications  about the conjecture.


\begin{thebibliography}{99}
\bibitem{Computing} M. Beck and S. Robins,
\emph{{Computing the Continuous
Discretely}, 
 Springer-verlag}, 2006

\bibitem{poly}M. Beck, I. Gessel and T. Komatsu,
{The polynomial part of a restricted partition function related to
the Frobenius problem},\emph{ Electronic Journal of Combinatorics}
\textbf{8}(1) (2001), N7, 1-5.

\bibitem{number} W. Dahmen and C. Micchelli, 
{The Number of Solutions to Linear Diophantine Equations and
Multivariate Splines}, \emph{Transactions of the American
Mathematical Society},  308(1988), 509-532.


\bibitem{E} E. Ehrhart, Polynomes arithmetiques et Methode des Polyedres en Combinatoire, \emph{International Series of Numerical Mathematics}, Vol.
35, Birkhauser Verlag, Basel/Stuttgart, 1977.



\bibitem{mine} N. Li and S. Chen,  
{On the Ring of Integer-valued Quasi-polynomial}, arXiv:0706.4288v2
[math.NT].
\bibitem{ours2} N. Li and S. Chen,  
{On Popoviciu Type Formulas for Generalized Restricted Partition
Function}, arXiv:0709.3571v1 [math.NT].

\bibitem{stum}J. De Loera and B. Sturmfels, Algebraic Unimodular Counting,
\emph{ Mathematical Programming, Series B}, \textbf{96} (2003)
183-203, doi:10.1007/s10107-003-0383-9

\bibitem{Bhu} B. Mishra,  
\emph{{Algorithmic Algebra},  
  Springer-Verlag, }  
2001. 
\bibitem{stanley}  R. Stanley,   
\emph{{Enumerative Combinatorics},   Vol.  1.   Cambridge University
Press}, 1996.


\bibitem{vector}B. Sturmfels, On Vector Partition Functions, \emph{Journal
of Combinatorial Theory, Series A }\textbf{72}(1995), 302-209,
doi:10.1016/0097-3165(95)90067-5

\bibitem{Xu} Z. Xu, 
{Multi-dimensional Versions of a Formula of Popoviciu},\emph{ Science in China Series A.} \textbf{49}, (2006) 75-82, doi: 10.1007/s11425-005-0036-y

\bibitem{Xu2}Z. Xu, An Explicit Formulation for Two Dimensional Vector Partition Functions,  to appear in \emph{Contemporary Math : Integer points in
polyhedra--Geometry, Number Theory, Representation Theory, Algebra,
Optimization, Statistics}.








\end{thebibliography}
\end{document}